\documentclass[a4paper,10pt]{amsart}
\usepackage{amssymb, diagrams,fullpage,xspace, mathrsfs}
\usepackage{hyperref}
\newcommand{\QQ}{\mathbf{Q}}
\newcommand{\ZZ}{\mathbf{Z}}

\DeclareMathOperator{\Fil}{Fil}
\DeclareMathOperator{\Spec}{Spec}
\newcommand{\dR}{\mathrm{dR}}
\newcommand{\cris}{\mathrm{cr}}
\newcommand{\st}{\mathrm{st}}
\newcommand{\DD}{\mathbf{D}}
\newcommand{\BB}{\mathbf{B}}
\newcommand{\Qp}{\QQ_p}

\newcommand{\rgam}{\mathbb{R}\Gamma}

\newcommand{\syn}{\mathrm{syn}}
\newcommand{\et}{\text{\upshape \'et}} 
\DeclareMathOperator{\tr}{tr}

\newcommand{\into}{\hookrightarrow}

\newtheorem{theorem}{Theorem}[section]
\newtheorem{proposition}[theorem]{Proposition}

\newtheorem{definition}[theorem]{Definition}
\newtheorem{assumption}[theorem]{Assumption}

\theoremstyle{remark}
\newtheorem{remark}[theorem]{Remark}

\newcommand{\Ne}{Nekov\'a\v{r}\xspace}
\newcommand{\Ni}{Nizio\l\xspace}

\title{Finite polynomial cohomology for general varieties}
\author{Amnon Besser}
\address{Department of Mathematics, Ben Gurion University, Be'er Sheva 84105, Israel}
\email{bessera@math.bgu.ac.il}

\author{David Loeffler}

\address[Loeffler]{Mathematics Institute\\
Zeeman Building, University of Warwick\\
Coventry CV4 7AL, UK}
\email{d.a.loeffler@warwick.ac.uk}

\author{Sarah Livia Zerbes}
\address[Zerbes]{Department of Mathematics \\
University College London\\
Gower Street, London WC1E 6BT, UK}
\email{s.zerbes@ucl.ac.uk}

\thanks{The authors' research was supported by the following grants: Royal Society University Research Fellowship (Loeffler); EPSRC First Grant EP/J018716/1 (Zerbes).}

\begin{document}

\begin{abstract}
 \Ne and \Ni have introduced in \cite{NN} a version of syntomic cohomology valid for arbitrary varieties over $p$-adic fields. This uses a mapping cone construction similar to the rigid syntomic cohomology of \cite{Bes98a} in the good-reduction case, but with Hyodo--Kato (log-crystalline) cohomology in place of rigid cohomology.

 In this short note, we describe a cohomology theory which is a modification of the theory of \Ne--\Ni, modified by replacing $1 - \varphi$ with other polynomials in $\varphi$. This is the analogue for bad-reduction varieties of the finite-polynomial cohomology of \cite{Bes97}; and we use this cohomology theory to give formulae for $p$-adic regulator maps, extending the results of \cite{Bes97, Bes98b, Bes10} to varieties over $p$-adic fields, without assuming any good reduction hypotheses.
\end{abstract}

\maketitle

 \section{Preliminaries from $p$-adic Hodge theory}

  \subsection{Filtered $(\varphi, N)$-modules and their cohomology}

   We recall some standard constructions for $(\phi, N)$-modules, following \S 2.4 of \cite{NN}.

   Let $K$ be a $p$-adic field (i.e.~the field of fractions of a complete DVR $V$ of mixed characteristic $(0, p)$, with perfect residue field $k$).

   We define a \emph{filtered $(\varphi, N, G_K)$-module} over $K$ to be the data of a finite-dimensional $K_0^{\mathrm{nr}}$-vector space $D$, where $K_0$ is the maximal unramified subfield of $K$ and $K_0^{\mathrm{nr}}$ its maximal unramified extension, equipped with the following structures: a $K_0^{\mathrm{nr}}$-semilinear, bijective Frobenius $\varphi$; a $K_0^{\mathrm{nr}}$-linear monodromy operator $N$ satisfying $N\varphi = p\varphi N$; a $K_0^{\mathrm{nr}}$-semilinear action of $G_K$; and a decreasing filtration of the $K$-vector space $D_K = (D \otimes_{K_0^{\mathrm{nr}}} \overline{K})^{G_K}$ by $K$-vector subspaces $\Fil^i D_K$. Such objects form an abelian category in the obvious way (with morphisms required to be strictly compatible with the filtration on $D_K$).

   \begin{theorem}[Colmez--Fontaine] The subcategory of ``weakly admissible'' filtered $(\varphi, N)$-modules is equivalent to the category of semistable $\Qp$-linear representations of $G_K$, via the functor $V \mapsto \DD_{\mathrm{pst}}(V) = \bigcup_{M} (V \otimes \BB_\st)^{G_M}$, where the direct limit is taken over finite extensions $M / K$.
   \end{theorem}

   Let $D$ be a filtered $(\varphi, N, G_K)$-module. We write $D_{\st} = D^{G_K}$, and define a complex $C_{\mathrm{st}}(D)$ by
   \[ D_{\st} \oplus \Fil^0 D_K \rTo D_{\st} \oplus D_{\st} \oplus D_K \rTo D_{\st},\]
   with morphisms given by $(u, v) \mapsto \left( (1 - \varphi)u, Nu, u - v\right)$ and $(w, x, y) \mapsto Nw - (1 - p\varphi) x$.

   \begin{theorem}[Bloch--Kato, \Ne]
    \label{thm:phinmodproperties}
    The complex $H^i_{\st}$ computes Ext groups in the category of filtered $(\varphi, N, G_K)$-modules; the subcategory of weakly admissible objects is closed under extensions, and the functor $\DD_{\st}$ induces functorial maps
    \[ H^i_{\st}(D) \to H^i(K, V). \]
    These maps are isomorphisms for $i = 0$ and injective for $i = 1$.
   \end{theorem}

  \subsection{Variants}

   We now construct a variant of the complex $C_{\st}(D)$ in which the semilinear Frobenius is replaced by a ``partially linearized'' one, and $1 - \varphi$ by a more general polynomial.

   We choose a finite extension $L / \Qp$ contained in $K$, and we write $f = [L_0 : \Qp]$ and $q = p^f$. We can then define $D_{\st, L} = D_{\st} \otimes_{L_0} L$, which we equip with an $L$-linear operator $\Phi$ given by extending scalars from the $L_0$-linear operator $\varphi^f$ on $D_{\st}$.

   Let $P \in 1 + T L[T]$ be a polynomial. We write $C_{\st, L, P}(D)$ for the complex of $L$-vector spaces
   \[ D_{\st, L} \oplus \Fil^0 D_K \rTo D_{\st, L} \oplus D_{\st, L} \oplus D_K \rTo D_{\st, L},\]
   with the maps given by $u,v \mapsto (P(\Phi) u, Nu, u - v)$ and $(w, x, y) \mapsto Nw - P(q\Phi)x$. Thus $C_{\st}(D) = C_{\st, \Qp, P_0}(D)$ where $P_0(T) = 1 - T$. We write $H^i_{\st, L, P}(D)$ for the cohomology of $C_{\st, L, P}(D)$.

   \begin{remark}
    In applications of the theory, we are almost always interested in the case when $K$ is finite over $\Qp$ and $L = K$; but in order to set up the theory we need a base-change compatibility which seems to be easier to prove for varying $K$ but fixed $L$, which is why we have set up the theory for $L \ne K$.
   \end{remark}

  \subsection{Cup products}
   \label{sect:cupprod1}

   \begin{definition}[{cf.~\cite[Definition 4.1]{Bes97}}]
    If $P, Q \in 1 + T L[T]$, then we define $P \star Q \in 1 + T L[T]$ as the polynomial with roots $\{ \alpha_i \beta_j\}$, where $\{\alpha_i\}$ and $\{\beta_j\}$ are the roots of $P$ and $Q$ respectively.
   \end{definition}

   \begin{proposition}
    There are cup-products $C^\bullet_{\st, L, P_1}(D_1) \times C^\bullet_{\st, L, P_2}(D_2) \to C^\bullet_{\st, L, P_1 \star P_2}(D_1 \otimes D_2)$, associative and graded-commutative up to homotopy.
   \end{proposition}

   \begin{proof}
    This is a straightforward exercise in homological algebra. Let $\lambda \in K$ and choose polynomials $a(T_1, T_2)$ and $b(T_1, T_2)$ such that $a(T_1, T_2) P_1(T_1) + b(T_1, T_2) P_2(T_2) = (P_1 \star P_2)(T_1 T_2)$. Then the cup-products in the various degrees are given by the following table:

    \begin{tabular}{|c|c|c|c|}
     \hline
      & $(u', v')$ & $(w', x', y')$ & $z'$ \\
     \hline
      $(u, v)$ & $(u \otimes u', v \otimes v')$ & $\left(\begin{array}{c} b(\Phi_1, \Phi_2)(u \otimes w'),\\ u \otimes x', \\ (\lambda u + (1 - \lambda) v) \otimes y' \end{array}\right)$ & $b(\Phi_1, q\Phi_2)(u \otimes z')$ \\ \hline
      $(w, x, y)$ & $\left(\begin{array}{c} a(\Phi_1, \Phi_2)(w \otimes u'),\\ x \otimes u', \\ y \otimes ((1 - \lambda) u' + \lambda v') \end{array}\right)$ & $-a(\Phi_1, q\Phi_2)(w \otimes x') + b(q\Phi_1, \Phi_2)(x \otimes w')$ & 0 \\ \hline
      $z$ & $a(q\Phi_1, \Phi_2)(z \otimes u')$ & 0 & 0 \\ \hline
    \end{tabular}

    One verifies easily that changing the value of $\lambda$, or the polynomials $a$ and $b$, changes the product by a chain homotopy.
   \end{proof}

  \subsection{Convenient modules}
   \label{sect:convenient}

   \begin{definition}
    Let $D$ be a filtered $(\varphi, N, G_K)$-module. We say $D$ is \emph{convenient} (for some choice of $L$ and $P$) if $N = 0$ and $P(\Phi)$ and $P(q\Phi)$ are bijective as endomorphisms of $D_{\st, L}$.
   \end{definition}

   Note that $D$ is convenient if and only if $D^*(1)$ is convenient. If $D$ is convenient, then the inclusion $D_K \into C^1_{\st}(D)$ induces an isomorphism
   \begin{equation}
    \label{eq:convenient}
    \frac{D_K}{\Fil^0 D_K} \cong H^1_{\st}(D).
   \end{equation}

   The inverse of this isomorphism is given by
   \begin{equation}\label{difference}
    (w, x, y) \mapsto y - \iota(P(\Phi)^{-1}w) \bmod \Fil^0,
   \end{equation}
   where $\iota$ is the natural inclusion $D_{\st, L} \into D_K$; note that we must have $Nw = P(q\Phi)x$ since $(w, x, y)$ is a cocycle, but by assumption $N = 0$ and $P(q\Phi)$ is bijective, so in fact we have $x = 0$. Note that this map commutes with change of $P$ and of $L$ (where defined).

   The following special case will be of importance below:

   \begin{proposition}
    \label{prop:Qp1}
    If $P(1) \ne 0$ and $P(q^{-1}) \ne 0$, then there is an isomorphism
    \[ H^1_{\st}(\Qp(1)) \cong K.\]
   \end{proposition}

   The key to our description of syntomic regulators is the following. We now take $L = K$.

   \begin{proposition}
    \label{prop:cupprod1}
    Suppose that $D$ is crystalline (i.e.~$N = 0$ on $D_{\st}$ and $G_K$ acts trivially), and that $D$ is convenient. Then for any $\lambda \in \Fil^0 D^*(1)_K = (D_K / \Fil^0)^*$, there is a polynomial $Q$ such that $\lambda \in H^0_{\st, K, P}(D^*(1))$ and $(P \star Q)(1) \ne 0$, $(P \star Q)(q^{-1}) \ne 0$; and if $P$ is such a polynomial, then we have a commutative diagram
    \begin{diagram}
     D_K / \Fil^0 & \rTo_{\cong}&  H^1_{\st, K, P}(D) \\
     \dTo^{\lambda} & & \dTo_{\cup\lambda} \\
     K & \rEq & H^1_{\st, K, P \star Q}(\Qp(1))
    \end{diagram}
    where the right-hand vertical map is the cup-product of the previous section.
   \end{proposition}

   \begin{remark}
    If we take $\lambda = 1$ in the formulae for the cup-product pairing, we see that the cup product $C^1_{\st, K, P} \times C^0_{\st, K, P} \to C^1_{\st, K, P}$ restricted to the direct summands $D_K \subseteq C^1_{\st, K, P}$ and $\Fil^0 D_K \subseteq C^0_{\st, K, P}$ is just the natural tensor product.
   \end{remark}

   \begin{remark}
    The utility of this remark is that if $D$ is convenient, it allows us to completely describe an element of $H^1_{\st, L, P}(D)$ via its cup-products with elements of $H^0_{\st, L, Q}(D^*(1))$ for suitable polynomials $Q$. Note that even if $P$ is some very simple specific polynomial such as $1 - T$, we still need to be able to choose $Q$ freely for this to work.
   \end{remark}

\section{Finite polynomial cohomology for general varieties}

 \subsection{Summary of the theory of \Ne--\Ni}

  We briefly summarize some of the main results of the paper \cite{NN}. Let $\operatorname{Var}(K)$ be the category of varieties over $K$ (i.e.~reduced separated $K$-schemes of finite type).

  \begin{theorem}
   \label{thm:NNtheory1}
   There exists a functor $X \mapsto \rgam_{\syn}(X_h, *)$ from $\operatorname{Var}(K)$ to the category of graded-commutative differential graded $E_\infty$-algebras over $\Qp$, equipped with period morphisms
   \[\rho_{\syn}: \rgam_{\syn}(X_h, r) \to \rgam_{\et}(X_{K, \et}, \Qp(r)). \]
   The cohomology theory $X \mapsto \rgam_{\syn}(X_h, r)$ has pushforward maps for projective morphisms, and has a functorial map from Voevodsky's motivic cohomology, compatible with the \'etale realization map via $\rho_{\syn}$.
  \end{theorem}

  For $X$ a variety over $K$, let $H^j_{HK}(X_h)$ and $H^j_{\dR}(X_h)$ be the extensions of Hyodo--Kato and de Rham cohomologies defined by Beilinson, and $\iota_{\dR}^B: H^j_{HK}(X_h) \otimes_{K_0} K \to H^j_{\dR}(X_h)$ the comparison isomorphism relating them. One knows (by the results of Beilinson cited in \cite[\S 4]{NN}) that we have
  \[ H^j_{HK}(X_h) = D^{j}(X_h)_{\st}, \quad H^j_{\dR}(X_h) = D^{j}(X_h)_{K}\]
  where $D^{j}(X_h)$ is the filtered $(\varphi, N, G_K)$-module $\DD_{\mathrm{pst}}(H^j(X_{\overline{K}, \et}))$.

  \begin{theorem}
   \label{thm:NNtheory2}
   There is a ``syntomic descent spectral sequence''
   \[ H^i_{\st}(D^{j}(X_h)(r)) \Rightarrow H^{i + j}_{\syn}(X_h, r),\]
   and the morphisms
   \[ H^i_{\st}(D^{j}(X_h)(r)) \to H^i(K, H^j_{\et}(X_{\overline{K}, \et}, \Qp(r))\]
   from Theorem \ref{thm:phinmodproperties} assemble into a morphism from the syntomic spectral sequence to the Hochschild--Serre spectral sequence
   \( H^i(K, H^j_{\et}(X_{\overline{K}, \et}, \Qp(r))) \Rightarrow H^{i+j}_{\et}(X_{K, \et}, \Qp(r)),\)
   compatible with the period morphism $\rho_{\syn}$ on the abutment.
  \end{theorem}

 \subsection{Definition of $P$-syntomic cohomology}

 We now develop a theory which very closely imitates that of \cite{NN}, but modified to use general polynomials in Frobenius in the place of $1 - \varphi$, and replacing the $p$-power Frobenius $\varphi$ with a ``partially linearized'' Frobenius. We choose a finite extension $L / \Qp$ is a finite extension contained in $K$, with residue class degree $f = [L_0 : \Qp]$, and $P \in 1 + T L[T]$.

  Let $(U, \overline{U})$ be an ``arithmetic pair'' (in the sense of \emph{op.cit.}), log-smooth over $V^\times$ (i.e.~$\Spec(V)$ with the log structure associated to the closed point).

  Let $\rgam_{\cris}(U, \overline{U})_{\QQ}$ be the rational crystalline cohomology (defined as
  \[ \Qp \otimes_{\ZZ_p} \operatorname{holim}_n \rgam(\overline{U}_{\et}, Ru_{U_n^\times / W_n(k) *} \mathscr{O}_{U_n^\times / W_n(k)}),\]
  cf.~\S 3.1 of \emph{op.cit.}). This is a complex of $K_0$-vector spaces, hence of $L_0$-vector spaces, and we write $\rgam_{\cris}(U, \overline{U})_{L} = L \otimes_{L_0} \rgam_{\cris}(U, \overline{U})_{\QQ}$. There is a $K_0$-semilinear Frobenius $\varphi$ on $\rgam_{\cris}(U, \overline{U})_{\QQ}$, and we let $\Phi = \varphi^f$, which is $L_0$-linear and thus extends to an $L$-linear operator on $\rgam_{\cris}(U, \overline{U})_{L}$. For $r \ge 0$ we write $\Phi_r = q^{-r} \Phi = (p^{-r} \varphi)^f$.

  \begin{remark}
   Curiously, the complex $\rgam_{\cris}(U, \overline{U})_{\QQ}$ is already a complex of $K$-vector spaces, but this $K$-linear structure interacts very badly with the Frobenius, and passing to Frobenius eigenspaces ``kills off'' all contributions from $K \setminus K_0$. So we must ``add the $K$-linear structure a second time'' in order to obtain a $K$-linearized theory of syntomic cohomology.
  \end{remark}

  For any $r \ge 0$, the quasi-isomorphism $\gamma_r^{-1} : \rgam_{\cris}(U, \overline{U}, \mathscr{O} / \mathscr{J}^{[r]})_{\QQ} \rTo^\sim \rgam_{\dR}(U, \overline{U}_K) / \Fil^r$ gives rise to a morphism (not, of course, a quasi-isomorphism)
  \[ \gamma_{r, L}^{-1}: \rgam_{\cris}(U, \overline{U})_{L} \to \rgam_{\dR}(U, \overline{U}_K) / \Fil^r.\]

  \begin{definition}
   We define
   \[ \rgam_{\syn, L, P}(U, \overline{U}, r) = \left[\rgam_{\cris}(U, \overline{U})_{L} \rTo^{P(\Phi_r), \gamma_{r, L}^{-1}} \rgam_{\cris}(U, \overline{U})_{L} \oplus \rgam_{\dR}(U, \overline{U}_K) / \Fil^r\right] \]
   (where the brackets signify mapping fibre, in the $\infty$-derived category of abelian groups).
  \end{definition}

  \begin{remark}
   Note that the use $P(\Phi_r)$, rather than $P(\Phi)$; this choice gives somewhat cleaner formulations of some results (e.g.~the pushforward maps and the syntomic descent spectral sequence), but has the disadvantage of introducing a notational discrepancy between $\rgam_{\syn, L, P}(U, \overline{U}, r)$ and the finite-polynomial cohomology of \cite{Bes97}.
  \end{remark}

  The complexes $\rgam_{\syn, L, P}(U, \overline{U}, r)$ are complexes of $L$-vector spaces, functorial in pairs $(U, \overline{U})$, equipped with $L$-linear cup-products
  \[ \rgam_{\syn, L, P}(U, \overline{U}, r) \times \rgam_{\syn, L, Q}(U, \overline{U}, s) \rTo \rgam_{\syn, L,  P \star Q}(U, \overline{U}, r+s) \]
  which are associative and graded-commutative, up to coherent homotopy (i.e.~the direct sum
  \[ \bigoplus_{r, P}\rgam_{\syn, L, P}(U, \overline{U}, r)\]
  is an $E_\infty$-algebra over $L$).

  \begin{remark}
   Here $P \star Q$ is the convolution of the polynomials $P$ and $Q$. We give an explicit formula for the cup product in Proposition \ref{prop:cupprod} below.
  \end{remark}

  The following properties are immediate:

  \begin{proposition}
   \label{prop:fptheoryproperties}
   \begin{enumerate}
    \item (Compatibility with syntomic cohomology) If $L = \Qp$ and $P(T)$ is the polynomial $1 - T$, then $\rgam_{\syn, L, P}(U, \overline{U}, r)$ coincides with $\rgam_{\syn}(U, \overline{U}, r)_{\QQ}$ as defined in \cite{NN}.
    \item (Change of $P$) If $P, Q$ are two polynomials, there is a map
    \[ \rgam_{\syn, L, P}(U, \overline{U}, r)_{\QQ} \to \rgam_{\syn, L, PQ}(U, \overline{U}, r)_{\QQ}\]
    functorial in $(U, \overline{U})$ and compatible with cup-products, which is given by the diagram
    \[\begin{diagram}
     \rgam_{\cris}(U, \overline{U})_{L} &\rTo^{P(\Phi_r), \gamma_{r, L}^{-1}} &\rgam_{\cris}(U, \overline{U})_{L} \oplus \rgam_{\dR}(U, \overline{U}_K) / \Fil^r\\
     \dTo^{\mathrm{id}} & & \dTo_{(Q(\Phi_r), \mathrm{id})} \\
     \rgam_{\cris}(U, \overline{U})_{L} &\rTo^{PQ(\Phi_r), \gamma_{r, L}^{-1}} &\rgam_{\cris}(U, \overline{U})_{L} \oplus \rgam_{\dR}(U, \overline{U}_K) / \Fil^r
    \end{diagram}\]
    \item (Change of L) Let $L' / L$ be a finite extension, and suppose that we have $P(T) = P'(T^d)$ for some polynomial $P'$, where $d = [L'_0 : L_0] = f'/f$. Then there is a natural map
    \[ \rgam_{\syn, L, P}(U, \overline{U}, r) \rTo \rgam_{\syn, L', P'}(U, \overline{U}, r),\]
    functorial in $(U,\overline{U})$ and compatible with cup-products. If $L' / L$ is an unramified extension, this is an isomorphism.
   \end{enumerate}
  \end{proposition}

  \begin{proof}
   (1) and (2) are obvious by construction. For (3), note that we have $\Phi'_r = (p^{-r} \varphi)^{f'}$ and $\Phi_r = (p^{-r} \varphi)^{f}$, so $P(\Phi_r)=P'(\Phi'_r)$. Hence we have a commutative diagram
    \[\begin{diagram}
     \rgam_{\cris}(U, \overline{U})_{L} &\rTo^{P(\Phi_r), \gamma_{r, L}^{-1}} &\rgam_{\cris}(U, \overline{U})_{L} \oplus \rgam_{\dR}(U, \overline{U}_K) / \Fil^r\\
     \dTo & & \dTo \\
     \rgam_{\cris}(U, \overline{U})_{L'} &\rTo^{P'(\Phi'_r), \gamma_{r, L'}^{-1}} &\rgam_{\cris}(U, \overline{U})_{L'} \oplus \rgam_{\dR}(U, \overline{U}_K) / \Fil^r
    \end{diagram}\]
    where the vertical maps are the natural maps induced from the inclusion $L\subset L'$.

    If $L'/L$ is unramified, so $L'=LL_0'$, then we have a canonical isomorphism of complexes of $L_0'$-vector spaces
    \[ L \otimes_{L_0} \rgam_{\cris}(U, \overline{U})_{\QQ}\cong L' \otimes_{L_0'} \rgam_{\cris}(U, \overline{U})_{\QQ},\]
    so the vertical maps are isomorphisms.
  \end{proof}

  We will be particularly interested in a special case of this:

  \begin{proposition}
   If $P(1) = 0$ then there is a morphism of complexes
   \[ \rgam_{\syn}(U, \overline{U}, r)_{\QQ} \to \rgam_{\syn, L, P}(U, \overline{U}, r),\]
   functorial in $(U,\overline{U})$ and compatible with cup-products.
  \end{proposition}

  \begin{proof}
   This is built up by combining all three parts of the above proposition. Firstly, (1) identifies \Ne--\Ni's cohomology $\rgam_{\syn}(U, \overline{U}, r)_{\QQ}$ with the special case $P(T) = 1 - T, L = \Qp$ of our construction. If $P$ is any polynomial with $P(1) = 0$, then $P(T^f)$ is divisible by $1 - T$, so (2) gives a morphism
   \[ \rgam_{\syn, \Qp, 1-T}(U, \overline{U}, r) \to \rgam_{\syn, \Qp, P(T^f)}(U, \overline{U}, r)_{\Qp}.\]
   Finally (3) gives a morphism
   \[ \rgam_{\syn, \Qp, P(T^f)}(U, \overline{U}, r) \to \rgam_{\syn, L, P}(U, \overline{U}, r).\]
   All of these are visibly functorial in pairs $(U, \overline{U})$.
  \end{proof}

  We now show a relation analogous to Proposition 3.7 of \emph{op.cit.}. Suppose that $(U, \overline{U})$ is of Cartier type, so Hyodo--Kato cohomology is defined. We use Beilinson's variant of Hyodo--Kato cohomology, which has the advantage of having a comparison map $\iota_{\dR}^B: \rgam_{HK}^B(U, \overline{U})_{K} \rTo^{\sim} \rgam_{\dR}(U, \overline{U}_K)$ which is defined at the level of complexes and does not depend on making a choice of uniformizer of $K$.

  \begin{proposition}
   \label{prop:quasiisom}
   For each uniformizer $\pi$ of $V$, there is a quasi-isomorphism
   \[ \rgam_{\syn, L, P}(U, \overline{U}, r) \rTo_\sim^{\alpha'_{\pi, P, L}}
    \left[
    \begin{diagram}
     \rgam_{HK}^B(U, \overline{U})_{L} & \rTo^{(P(\Phi_r), \iota_{\dR}^B)}  & \rgam_{HK}^B(U, \overline{U})_{L} \oplus \rgam_{\dR}(U, \overline{U}_K) / \Fil^r \\
     \dTo^{N} & & \dTo^{(N, 0)}\\
     \rgam_{HK}^B(U, \overline{U})_{L} & \rTo^{P(\Phi_{r-1})} & \rgam_{HK}^B(U, \overline{U})_{L}
    \end{diagram}\right]
   \]
   where $q = p^f$.
  \end{proposition}

  \begin{remark}
   This map does actually depend on the choice of a uniformizer $\pi$, although its source and target are independent of any such choice.
  \end{remark}

  \begin{proof}
   We follow exactly the same argument as the proof given in \emph{op.cit.}. It suffices to check that $P(\Phi)$ is invertible on $I \otimes_{W(k)} H^i_{HK}(U, \overline{U})_{L}$, for any $P \in 1 + T L[T]$, where $I$ is the ideal in a divided power series ring over $W(k)$ considered in \emph{op.cit.}. We note that since $P$ is monic the formal Laurent series $1/P(T) = \sum_{n \ge 0} a_n T^n$ has positive radius of convergence, so there is some $A$ such that $\operatorname{ord}_p(a_n) \ge -nA$ for $n \gg 0$. This implies the convergence of the series $\sum a_n \Phi^n$ on $I \otimes_{W(k)} H^i_{HK}(U, \overline{U})_{L}$, which gives an inverse of $P(\Phi)$.
  \end{proof}

  Continuing to follow \cite{NN}, we have

  \begin{proposition}
   If $(T, \overline{T})$ is the base-change of $(U, \overline{U})$ to a finite Galois extension $K' / K$ with Galois group $G$, then the natural map $f: (T, \overline{T}) \to (U, \overline{U})$ induces a quasi-isomorphism
   \[ f^*: \rgam_{\syn, L, P}(U, \overline{U}, r) \rTo^\sim \rgam(G, \rgam_{\syn, L, P}(T, \overline{T}, r)). \]
  \end{proposition}

  \begin{proof}
   This follows by exactly the same proof as in \emph{op.cit.}: the proof proceeds by constructing compatible maps between the various variants of crystalline, de Rham, and Hyodo--Kato cohomology, and these all remain compatible after extending scalars to $L$.
  \end{proof}

  \begin{remark}
   We are principally interested in the case $L = K$, but it seems to be easier to prove base-change compatibility in $K$ for a fixed $L$. It seems eminently natural that if $K' / K$ is totally ramified and both fields are finite over $\Qp$ then we should get a quasi-isomorphism
   \[ \rgam_{\syn, P}(U, \overline{U}, r)_{K} \rTo^\sim \rgam(G, \rgam_{\syn, P}(T, \overline{T}, r)_{K'}),\]
   but this does not seem to be so easy to prove, and it is not needed for the calculations below.
  \end{remark}

 \subsection{h-sheafification}

  We now sheafify in the $h$-topology. We write $\mathscr{S}_{L, P}(r)$ for the sheafification of $(U, \overline{U}) \mapsto \rgam_{\syn, L, P}(U, \overline{U}, r)$, and we define
  \[ \rgam_{\syn, L, P}(X_h, r) = \rgam(X_h, \mathscr{S}_{L, P}(r)).\]

  \begin{proposition}
   For any arithmetic pair $(U, \overline{U})$ that is fine, log-smooth over $V^\times$ and of Cartier type, the canonical maps
   \[ \rgam_{\syn, L, P}(U, \overline{U}, r) \to \rgam_{\syn, L, P}(U_h, r) \]
   are quasi-isomorphisms.
  \end{proposition}

  \begin{proof}
   This is exactly the generalization to our setting of Proposition 3.16 of \emph{op.cit.}. The long and highly technical proof fortunately carries over verbatim to general $P$.
  \end{proof}

  \begin{proposition}
   There is a ``$P$-syntomic descent spectral sequence''
   \[ E_2^{ij} = H^i_{\st, L, P}(D^j(r)) \Rightarrow H^{i+j}_{\syn, L, P}(X_h, r), \]
   where $D^j(r) = \DD_{\mathrm{pst}}(H^j(X_{\overline{K}, \et}, \Qp(r)))$ as above. This spectral sequence is compatible with extension of $L$ (where defined) and change of $P$. Moreover, it is compatible with cup-products, where the cup-product on the terms $E_2^{ij}$ is given by the construction of \S \ref{sect:cupprod1}.
  \end{proposition}

  \begin{proof}
   For the existence of the spectral sequence, see Proposition 3.17 of \emph{op.cit.}. The compatibility with cup-products is immediate from the definition of the cup-product on $P$-syntomic cohomology (see \S \ref{sect:cupprod} below for explicit formulae).
  \end{proof}

  Note that the spectral sequence implies that the cohomology groups $H^{i}_{\syn, L, P}(X_h, r)$ are zero if $i > 2\dim X + 2$, and if $K / \Qp$ is a finite extension, then the cohomology groups are finite-dimensional $L$-vector spaces (with dimension bounded independently of $P$). Moreover, the spectral sequence obviously degenerates at $E_3$, and this gives a 3-step filtration on the groups $H^{i}_{\syn, L, P}(X_h, r)$:

  \begin{definition}
   We write $\Fil^m H^{i}_{\syn, L, P}(X_h, r)$ for the 3-step decreasing filtration on $H^{i}_{\syn, L, P}(X_h, r)$ induced by the $P$-syntomic spectral sequence. Concretely, we have
   \begin{align*}
    \Fil^0 / \Fil^1 &= \ker \left( H^i_{HK}(X_h)_L^{P(\Phi_r) = 0, N = 0} \cap \Fil^r H^i_{\dR}(X_h) \rTo \frac{H^{i-1}_{HK}(X_h)_L}{\operatorname{im} P(\Phi_{r-1}) + \operatorname{im}(N)}\right), \\
    \Fil^1 / \Fil^2 &= \frac{ \left\{ (x, y, z) \in H^{i-1}_{HK}(X_h)_L \oplus H^{i-1}_{HK}(X_h)_L \oplus H^{i-1}_{\dR}(X_h) / \Fil^r : Nx = P(\Phi_{r-1})y\right\}}
    {\left\{(P(\Phi_r) x, Nx, \iota_{\dR}(x): x \in H^{i-1}_{HK}(X_h)_L\right\}},\\
    \Fil^2 / \Fil^3 &= \operatorname{coker}\left( H^{i-1}_{HK}(X_h)_L^{P(\Phi_r) = 0, N = 0} \cap \Fil^r H^{i-1}_{\dR}(X_h) \rTo \frac{H^{i-2}_{HK}(X_h)_L}{\operatorname{im} P(\Phi_{r-1}) + \operatorname{im}(N)}\right).
   \end{align*}
  \end{definition}

  \begin{remark}
   It seems natural to conjecture that the ``knight's move'' maps
   \[ H^0_{\st, L, P}(D^j(X_h)(r)) \to H^2_{\st, L, P}(D^{j-1}(X_h)(r))\]
   should be zero for smooth proper $X$, extending the conjecture for syntomic cohomology formulated in Remark 4.10 of \cite{NN}. We do not know if this conjecture holds in general, but the applications below will all concern cases where either the source or the target of this map is zero.
  \end{remark}


  \subsection{Explicit cup product formulae}
   \label{sect:cupprod}

  We now use Proposition \ref{prop:quasiisom} to give an explicit formula for the cup product
  \[ H^i_{\syn, L, P}(X_h, r) \times H^j_{\syn, L, Q}(X_h,s) \rTo H^{i+j}_{\syn, L, P\star Q}(X_h,r+s), \]
  which generalizes the description of the cup-product on the complexes $C^\bullet_{\st, L, P}(D)$ given in the previous section.

  A class $\eta\in H^i_{\syn,P}(X_h, r)$ is represented by the following data:
  \begin{align*}
   u&\in \rgam_{HK}^{B,i}(X_h)_{L},& v &\in \Fil^r \rgam_{\dR}^{i}(X_h), \\
   w,x&\in  \rgam_{HK}^{B,i-1}(X_h)_{L},& y &\in \rgam_{\dR}^{i-1}(X_h)_{L}, \\
   z&\in \rgam_{HK}^{B,i-2}(X_h)_{L}
  \end{align*}
  which satisfy the relations
  \begin{align*}
   \mathrm{d} u &= 0, & \mathrm{d}v &= 0,\\
   dw &=P(\Phi)u, &  dx & =N u, & \mathrm{d}y &= \iota^B_{\dR}(u) - v,\\
   dz & =Nw- P(q\Phi)x.
  \end{align*}

  Let $\eta'=[u', v'; w',x',y'; z']$ be a corresponding representation of $\eta' \in H^j_{\syn, L, Q}(X_h,s)$. We want to find explicit formulae for the class $\eta''=\eta\cup\eta'$.

  As before, fix polynomials $a(t,s)$ and $b(t,s)$ such that
  \[ (P\star Q)(ts)=a(t,s)P(t)+b(t,s)Q(s),\]
  and $\lambda \in K$.

  \begin{proposition}
   \label{prop:cupprod}
   The cup-product is given by $\eta''=[u'', v''; w'', x'', y''; z'']$, where
   \begin{align*}
    u'' & = u \cup u',\\
    v'' &= v \cup v',\\
    w'' & = a(\Phi,\Phi)(w\cup u')+(-1)^i b(\Phi,\Phi)(u\cup w'),\\
    x'' & = (x\cup u')+(-1)^i (u \cup x'),\\
    y'' &= y \cup \iota_{\dR}^B\Big((1 - \lambda) u' + \lambda v'\Big) + (-1)^i \iota_{\dR}^B\Big( \lambda u + (1 - \lambda) v\Big) \cup y',\\
    w'' & = a(q\Phi,\Phi)(z \cup u') - (-1)^{i} a(\Phi,q\Phi)(w \cup x') + (-1)^i b(q\Phi,\Phi)(x\cup w') - b(\Phi,q\Phi)(u \cup z').
   \end{align*}
   Here, we write e.g.~$a(\Phi,\Phi)(w\cup u')$ for the image of $a(\Phi_1,\Phi_2)(w \otimes u') \in \rgam_{HK}^{B, i}(X_h)_L \otimes \rgam^{B, j-1}_{HK}(X_h)_L$ under the cup product map into $\rgam_{HK}^{B, i + j - 1}(X_h)_L$, where $\Phi_1$ and $\Phi_2$ are the $q$-power Frobenius maps of the two factors.
  \end{proposition}

  \begin{proof}
   This is simply a translation of the formulae of \S \ref{sect:cupprod1} into the setting of complexes.
  \end{proof}


 \subsection{Pushforward maps and compact supports}

  We now extend to our setup the constructions of D\'eglise's Appendix B to \cite{NN}.

  \begin{theorem}
   \label{thm:pushfwd}
   Let $f: X \to Y$ be a smooth proper morphism of smooth $K$-varieties, of relative dimension $d$. Then there are pushforward maps
   \[ f_*: H^i_{\syn, L, P}(X_h, r) \to H^{i + 2d}_{\syn, L, P}(X_h, r + d),\]
   which are compatible with change of $L$ and change of $P$ in the obvious sense.
  \end{theorem}

  \begin{proof}
   Each of the sheaves $\mathscr{S}_{L, P}(r)$ is $h$-local, by definition. They are also $\mathbf{A}^1$-local: this is virtually immediate from the corresponding result for the underlying Hyodo--Kato and de Rham cohomology theories, as in Prop 5.4 of \emph{op.cit.}. The same methods also yield the projective bundle theorem.

   Hence we can consider the direct sum of $h$-sheaves given by
   \[ \mathscr{S}_{L, P} := \bigoplus_{r} \mathscr{S}_{L, P}(r). \]
   Cup-product gives us maps $\mathscr{S}_L(1) \otimes \mathscr{S}_{L, P}(r) \to \mathscr{S}_{L, P}(r+1)$, where $\mathscr{S}_L(1)$ is defined using the polynomial $1 - T$; this gives the direct sum $\mathscr{S}_{L, P}$ the structure of a Tate $\Omega$-spectrum. The same argument as in \emph{op.cit.} now gives pushforward maps (and the projection formula for these holds by construction). It is clear that this argument is compatible with change of $L$ and of $P$.
  \end{proof}

  The same argument also gives compactly-supported cohomology complexes:

  \begin{theorem}
   There exist compactly-supported cohomology complexes $\rgam_{\syn, L, P, c}(X_h, r)$, contravariantly functorial with respect to proper morphisms and covariantly functorial (with a degree shift) with respect to smooth morphisms; and there is a functorial morphism
   \[ \rgam_{\syn, L, P, c}(X_h, r) \to \rgam_{\syn, L, P}(X_h, r)\]
   which is an isomorphism for proper $X$.

   The compactly-supported cohomology has a descent spectral sequence
   \[ E^2_{ij} = H^i_{\st, L, P}(D^j_c(X_h)(r)) \Rightarrow H^{i + j}_{\syn, L, P, c}(X_h, r)\]
   where $D^j_c(X_h) = \DD_{\mathrm{pst}}(H^i_{\et, c}(X_{\overline{K}, \et}, \Qp))$.
  \end{theorem}

  Finally, we have a projection formula:

  \begin{theorem}
   \label{thm:cupprod}
   There are cup-products
   \[ \rgam_{\syn, L, P, c}(X_h, r) \times \rgam_{\syn, L, Q}(X_h, s) \to \rgam_{\syn, L, P\star Q, c}(X_h, r + s), \]
   and for $f$ a smooth proper morphism, we have the projection formula
   \[ f_*(\alpha) \cup \beta = f_*(\alpha \cup f^* \beta). \]
  \end{theorem}

  \begin{proof}
   This follows from the construction of the pushforward map, cf.~\cite{DM12}.
  \end{proof}


 \section{Application to computation of regulators}


 We'll now apply the constructions above to give formulae describing the syntomic regulator map $H^i_\mathrm{mot}(X, j) \to H^i_{\syn}(X_h, j)$, where $X$ is a product of copies of an affine curve $Y$. More specifically, we will give formulae in the following cases:
 \begin{itemize}
  \item $X = Y$ is a single curve and we are given a class in $H^2_\mathrm{mot}(X, 2)$ that is the cup product of two units on $Y$;
  \item $X = Y^2$ is the product of two curves, and we are given a class in $H^3_\mathrm{mot}(X, 2)$ that is the pushforward of a unit along the diagonal embedding $Y \into Y^2$;
  \item $X = Y^3$ is the product of three curves and we consider the class in $H^4_\mathrm{mot}(X, 2)$ given by the cycle class of the diagonal $Y \into Y^3$.
 \end{itemize}
 In the case where $Y$ has a smooth model over $\ZZ_p$ these regulators have been described using finite-polynomial cohomology (in \cite{Bes98b}, \cite{Bes10}, and \cite{Bes97} respectively); and these have been applied in proving explicit reciprocity laws for Euler systems (in \cite{BD1,BDR1,DR1} respectively). We use the generalization of finite-polynomial cohomology to arbitrary varieties described in the preceding section to extend this description to the case of arbitrary smooth curves over $K$. However, our results are less complete, in that we only obtain a full description of the image of the regulator in a ``convenient'' quotient of the cohomology of $X$ in the sense of \S \ref{sect:convenient}. (In other words, we allow $X$ to have bad reduction, but we project to a quotient of its cohomology which looks like the cohomology of a variety with good reduction.)

 \subsection{Syntomic regulators}

  Now let us consider the following general setting: $X$ is a smooth connected $d$-dimensional affine $K$-variety, and $z \in H^{d + 1}_\mathrm{mot}(X, j)$ for some $j$. We have an etale realization map
  \[ r_{\et}: H^{d + 1}_\mathrm{mot}(X, j) \rTo H^{d+1}_{\et}(X_{K, \et}, j).\]
  The Hochschild--Serre descent spectral sequence gives a map
  \[ H^{d + 1}(X_{K, \et}, \Qp(j)) \to H^0(K, H^{d + 1}(X_{\overline{K}, \et}, \Qp(j));\]
  but the latter group is zero (because an affine variety of dimension $d$ has \'etale cohomological dimension $d$) and thus we obtain a map
  \[ H^{d + 1}(X_{K, \et}, \Qp(j)) \to H^1(K, H^d(X_{\overline{K}, \et}, \Qp(j))).\]
  Theorem B of \cite{NN} shows that we have $r_{\et} = \rho_{\syn} \circ r_{\syn}$, where
  \[ r_{\syn}: H^{d + 1}_\mathrm{mot}(X, j) \rTo H^{d+1}_{\syn}(X_{K, h}, j)\]
  is the syntomic realization map, and
  \[ \rho_{\syn}: H^{d+1}_{\syn}(X_{K, h}, j) \to H^{d+1}_{\et}(X_{K, \et}, j) \]
  is the syntomic period morphism of Theorem \ref{thm:NNtheory1} above. Theorem \ref{thm:NNtheory2} shows that we have a diagram
  \begin{diagram}
   H^{d + 1}_\mathrm{syn}(X_{K, h}, j) & \rTo^{\rho_{\syn}} & H^{d + 1}_{\et}(X_{K, \et}, \Qp(j))\\
   \dTo & & \dTo\\
   H^1_{\st}(D^d(X_h)(j)) & \rTo^{\exp_{\st}} &  H^1(K, H^d(X_{\overline{K}, \et}, \Qp(j)))
  \end{diagram}
  where $D^d(X_h) = \DD_{\mathrm{pst}}\left(H^d(X_{\overline{K}, \et}, \Qp)\right)$, the vertical maps are induced by the syntomic and Hochschild--Serre descent spectral sequences, and the bottom horizontal map is the generalized Bloch--Kato exponential for the de Rham $G_K$-representation $H^d(X_{\overline{K}, \et}, \Qp(j))$.

  Now let $D$ be a quotient of $D^d(X_h)(j)$ (in the category of weakly-admissible $(\varphi, N, G_K)$-modules) which is crystalline and convenient, in the sense of \S \ref{sect:convenient} above. Then we have isomorphisms
  \[ H^1_{\st}(D) \cong H^1_{\st, K, 1 - T}(D) \cong D_K / \Fil^0 = (\Fil^0 D^*(1)_K )^*,\]
  where $D^*(1)$ is the Tate dual of $D$, which is a submodule of $D^d_{c}(X_h)(d + 1 - j)$. As we saw above, for any class $\eta \in \Fil^0 D^*(1)_K$, we may choose a polynomial $P$ such that $\eta \in H^0_{\st, K, P}(D^*(1))$ and $P(1) \ne 0$, $P(q^{-1}) \ne 0$; and the natural perfect pairing $\left(D_K / \Fil^0\right) \times \left(\Fil^0 D^*(1)_K\right) \to K$ coincides with the cup-product
  \[ H^1_{\st, K}(D) \times H^0_{\st, K, P}(D^*(1)) \to H^1_{\st, K, P}(\Qp(1)) \cong K.\]

  We now relate this to cup-products in $P$-syntomic cohomology.

  \begin{definition}
   If $P(1) \ne 0$ and $P(q^{-1}) \ne 0$, then we denote by
   \[ \tr_{X, \syn, K, P} : \frac{H^{2d + 1}_{\syn, X, P, c}(X_h, d + 1)}{\Fil^2} \rTo^\cong H^1_{\st, K, P}(\Qp(1)) \cong K \]
   the isomorphism given by the descent spectral sequence and Proposition \ref{prop:Qp1} above.
  \end{definition}

  \begin{theorem}
   \label{thm:ajcupprod}
   Let $D$ be a convenient crystalline quotient of $D^d(X_h)(j)$, and let $\eta \in \Fil^0 D^*(1)_K \subseteq \Fil^{d + 1 - j} H^d_{\dR, c}(X_h)$. If
   \[ \operatorname{pr}_D: H^{d + 1}_{\syn}(X_h, j) \rTo H^1_{\st}(D^d(X_h)(j)) \rTo H^1_{\st}(D) \cong D_K / \Fil^0\]
   is the natural projection, then for any $z \in H^{d + 1}_{\mathrm{mot}}(X, j)$, any polynomial $P$ such that $P(\Phi)(\eta) = 0$ and $P(1) \ne 0$, $P(q^{-1}) \ne 0$, and any  class
   \[ \tilde\eta \in H^d_{\syn, K, P, c}(X_h, d + 1 - j)\]
   lifting $\eta$, we have
   \[ \eta\left(\operatorname{pr}_D(r_{\syn}(z))\right) = \tr_{X, \syn, K, P}\left(r_{\syn, K}(z) \cup \tilde\eta\right)\]
   where $r_{\syn, K}(z)$ is the image of $r_\syn(z)$ under the natural map $H^1_{\st}(D^d(X_h)(j)) \to H^1_{\st, K, 1 - T}(D^d(X_h)(j))$.
  \end{theorem}

  \begin{proof}
   This is immediate from Proposition \ref{prop:cupprod1} and the compatibility of the $P$-syntomic descent spectral sequence with cup-products.
  \end{proof}

  Using the above formula together with the projection formula for cup-products (Theorem \ref{thm:cupprod}), we obtain the following consequences (in which $j = 2$ and $1 \le d \le 3$). For $\eta \in H^d_{\dR, c}(X / K)$ satisfying the hypotheses of Theorem 4.2, we write $\lambda_\eta$ for the map $H^{d + 1}_\syn(X_h, j) \to K$ given by the composition of $\operatorname{pr}_D$ and pairing with $\eta$.

  \begin{proposition}
   \label{prop:K0}
   If $d = 3$ and $z = \sum_i n_i [Z_i] \in H^{4}_{\mathrm{mot}}(X, 2)$ is the class of a codimension $2$ cycle with smooth components, and $\eta$ is as in Theorem \ref{thm:ajcupprod}, then we have
   \[ \lambda_\eta(r_{\syn}(z)) = \sum_i n_i \operatorname{tr}_{Z_i, \syn, K, P}( \iota_i^*(\tilde\eta)),\]
   where $\iota_i$ is the inclusion of $Z_i$ into $X$.
  \end{proposition}

  \begin{proof}
   Cf.~\cite[Theorem 1.2]{Bes97}. By the compatibility of motivic and syntomic pushforward maps, we have
   \[ r_{\syn, K}(z) = \sum_i n_i\ (\iota_i)_*\left(\mathbf{1}_{Z_i}\right),\]
   where $\mathbf{1}_{Z_i} \in H^0_{\syn, K}(Z_i, 0)$ is the identity class. By the projection formula, we have
   \[ r_{\syn, K}(z) \cup \tilde\eta = \sum_i n_i\, (\iota_i)_*\left( \mathbf{1}_{Z_i} \cup \iota_i^* \tilde\eta\right).\]
   Since the maps $\tr_{X, \syn, K, P}$ are compatible with pushforward (as is clear by comparison with their de Rham analogues), the result is now immediate from Theorem \ref{thm:cupprod}.
  \end{proof}

  \begin{proposition}
   \label{prop:K1}
   If $d = 2$ and $z = \sum_i (Z_i, u_i) \in H^{3}_{\mathrm{mot}}(X, 2)$, where $Z_i$ are codimension 1 cycles and $u_i \in \mathcal{O}(Z_i)^\times$, and $\eta$ is as in Theorem \ref{thm:ajcupprod}, then we have
   \[ \lambda_\eta(r_{\syn}(z)) = \sum_i n_i \operatorname{tr}_{Z_i, \syn, K, P}(r_{\syn, K}(u_i) \cup \iota_i^*(\tilde\eta)),\]
   where $\iota_i$ is the inclusion of $Z_i$ into $X$. Here, $r_{\syn}$ is the composition
   \[ \mathcal{O}(Z_i)^\times = H^1_{\mathrm{mot}}(Z_i,1)\rTo H^1_{\syn}(Z_i,1).\]
  \end{proposition}

  \begin{proof}
   Again, by the compatibility of motivic and syntomic pushforwards, we have
   \[ r_{\syn, K}(z) = \sum_i (\iota_i)_*(r_{\syn, K}(u_i)), \]
   and the result is immediate from Theorem \ref{thm:cupprod}.
  \end{proof}

  \begin{proposition}
   \label{prop:K2}
   If $d = 1$ and $z = \sum_i \{ u_i, v_i \} \in H^{2}_{\mathrm{mot}}(X, 2)$, where $\{ u_i, v_i \}$ is the Steinberg symbol of two units $u_i, v_i \in \mathcal{O}(X)^\times$, and $\eta$ is as in Theorem \ref{thm:ajcupprod}, then we have
   \[ \lambda_\eta(r_{\syn}(z)) = \sum_i \operatorname{tr}_{X, \syn, K, P}(r_{\syn, K}(u_i) \cup r_{\syn, K}(v_i) \cup \tilde\eta).\]
  \end{proposition}

  \begin{proof}
   Observe that the image of $\{ u_i, v_i \}$ in $H^2_{\syn}(X,2)$ is equal to $r_{\syn, K}(u_i) \cup r_{\syn, K}(v_i)$.
  \end{proof}

 \subsection{The cohomological triple symbol}

  We now define a ``triple symbol'' attached to three de Rham classes on a curve $Y$, which is closely related to the global triple index of \cite{Bes10}. Let $Y$ be a connected smooth affine curve over $K$.

  Let us choose a class $\eta \in H^1_{\dR, c}(Y / K)$; and let $\omega_1, \omega_2 \in \Fil^1 H^1_{\dR}(Y/K)$. We suppose that $\omega_1, \omega_2$ are in the image of the map
  \[ H^1_{HK}(Y_{K, h})^{N = 0} \otimes K  \into H^1_{\dR}(Y / K), \]
  and similarly for $\eta$ with compact supports; and that there are polynomials $P_1, P_2$ with the property that $P_1(q^{-1}\Phi)(\omega_1) = P_2(q^{-1}\Phi)(\omega_2) = 0$, where we regard $\Phi$ as an endomorphism of $H^1_{\dR}(Y/K)$ via the comparison isomorphism $\iota^B_{\dR}$.

  \begin{assumption}
   The following conditions are satisfied:
   \begin{itemize}
    \item There are polynomials $P_0, P_1, P_2 \in 1 + T K[T]$ such that
    \[ \eta \in H^0_{\st, K, P_0, c}(D^1_c(Y_{K, h})),\ \omega_1 \in H^0_{\st, K, P_1}(D^1(Y_{K, h})(1)),\ \omega_2 \in H^0_{\st, K, P_2}(D^1(Y_{K, h})(1)).\]
    \item The class $\omega_1$ is in the kernel of the ``knight's move'' map
    \[ H^0_{\st, K, P_1}(D^1(Y_{K, h})(1)) \to H^2_{\st, K, P_1}(D^0(Y_{K, h})(1)),\]
    and similarly for $\omega_2$.
    \item The polynomial $P_0 \star P_1 \star P_2$ does not vanish at $1$ or $q^{-1}$.
    \item We have $\eta \cup \omega_1 = \eta \cup \omega_2 = 0$ as elements of $H^2_{\dR, c}(Y / K) \cong K$.
   \end{itemize}
  \end{assumption}

  \begin{remark}
   We have
   \[ H^2_{\st, K, P}(D^0(Y_{K, h})(1)) = H^2_{\st, K, P}(\Qp(1)) = \frac{K}{P(1) K},\]
   so if $P_1(1) \ne 0$, then $\omega_1$ is automatically in the kernel of the ``knight's move'' map. In particular, this applies if $\omega_1$ is pure of weight $1$. On the other hand, if $\omega_1 = \operatorname{dlog} u$ for a unit $u$, then we have no choice but to take $P_1(1) = 0$, but the existence of the syntomic regulator $r_{\syn}$ and its compatibility with the de Rham regulator forces $\omega_1$ to be in the kernel of this map since we know it is in the image of $H^1_{\syn}(Y_h, 1)$.

   If all three de Rham classes are pure of weight 1, then $P_0 \star P_1 \star P_2$ has all its roots of weight $-1$, so in particular it is non-vanishing at $1$ and $q^{-1}$.
  \end{remark}

  \begin{definition}
   Under the above assumptions, we define the \emph{triple symbol} $[\eta; \omega_1, \omega_2] \in K$ by the formula
   \[ [\eta; \omega_1, \omega_2] = \tr_{Y, \syn, P_0 \star P_1 \star P_2}\left( \tilde \eta \cup \tilde\omega_1 \cup \tilde\omega_2\right)\]
   where $\tilde \eta \in H^1_{\syn, K, P_0, c}(Y_{K, h}, 0)$ and $\tilde\omega_i \in H^1_{\syn, K, P_i}(Y_{K, h}, 1)$ are liftings of $\eta$ and the $\omega_i$.
  \end{definition}

  \begin{proposition}
   The above quantity is independent of the choice of liftings and of the polynomials $P_i$.
  \end{proposition}

  \begin{proof}
   Firstly, we note that the natural map $H^1_{\syn, K, P_0, c}(Y_{K, h}, 0) \to H^1_{\dR, c}(Y / K)$ is an isomorphism, since the degree 0 compactly-supported cohomology is zero. Thus the class $\tilde\eta$ is uniquely defined.

   If $\tilde\omega_1$ is a class in $H^1_{\syn, K, P_1}$ lifting $\omega_1$, and $[u, v; w, x, y; z]$ is a representative of $\tilde\omega_1$ in $\rgam^1_{\syn, K, P_1}$ (with $z = 0$, necessarily, for degree reasons), then any other choice of lifting can be represented by $[u, v; w + \lambda, x + \mu, y; z]$ for some constants $\lambda, \mu \in K$ (with $\mu = 0$ unless $P_1(1) = 0$). From the definition of the cup-product, one sees that varying $\mu$ has no effect on $\tilde \eta \cup \tilde\omega_1 \cup \tilde\omega_2$; while varying $\lambda$ changes the cup product by a multiple of $\eta \cup \omega_2$, which is zero by assumption. Similarly, the assumption that $\eta \cup \omega_1 = 0$ implies that the cup-product is independent of the choice of lifting of $\omega_2$.

   So $[\eta; \omega_1, \omega_2]$ is well-defined for a fixed choice of polynomials $P_i$. However, both the cup-product and the map $\tr_{Y, \syn, P_0 \star P_1 \star P_2}$ are compatible with change of the polynomials $P_i$, so the symbol $[\eta; \omega_1, \omega_2]$ is also independent of these choices.
  \end{proof}

  \begin{remark}
   We may also carry out the same construction if $Y$ is projective, rather than affine, if we add the assumption that $P_0(1) \ne 0$ and $P_0(q) \ne 0$; this assumption implies that there is still a unique lifting of $\eta$ to $H^1_{\syn, K, P_0, c}(Y_{K, h}, 0) = H^1_{\syn, K, P_0}(Y_{K, h}, 0)$. In particular, this holds if $P_0$ is pure of weight 1.
  \end{remark}

 \subsection{Products of curves}

  We now assume $Y$ is a connected affine curve over $K$ with semistable reduction.    The following results are simply special cases of Propositions \ref{prop:K0}--\ref{prop:K2}, with the right-hand sides of the formulae rewritten in terms of the triple symbol.

  \begin{proposition}
   Let $\eta \in H^1_{\dR, c}(Y / K)$, and let $\omega_{1}, \omega_{2} \in \Fil^1 H^1_{\dR}(Y / K)$ be classes lying in the parabolic subspace (the image of $H^1_{\dR,c}$ in $H^1_{\dR}$). Suppose that the triple symbol $[\eta; \omega_1, \omega_2]$ is defined.

   Let $X = Y \times Y \times Y$ and let $z \in H^4_{\mathrm{mot}}(X, 2)$ be the class of the diagonal embedding $Y \into Y \times Y \times Y$. If $\mu$ denotes the class $\eta \otimes \omega_{1, c} \otimes \omega_{2, c} \in \Fil^2 H^3_{\dR, c}(X / K)$, where $\omega_{i, c}$ are any liftings of $\omega_i$ to compactly-supported cohomology, then
   \[ \lambda_{\mu}(r_{\syn}(z)) = [\eta; \omega_1, \omega_2]. \]
  \end{proposition}

  \begin{proposition}
   Let $\eta \in H^1_{\dR, c}(Y / K)$ and let $\omega \in \Fil^1 H^1_{\dR}(Y / K)$ be a class lying in the parabolic subspace. Let $u \in \mathcal{O}(Y)^\times$, and suppose that the triple symbol $[\eta; \omega_1, \operatorname{dlog} u]$ is defined.

   Let $X = Y \times Y$ and let $z \in H^3_{\mathrm{mot}}(X, 2)$ be the pushforward of $u \in H^1_\mathrm{mot}(Y, 1) \cong \mathcal{O}(Y)^\times$ along the diagonal embedding $Y \into Y \times Y$. If $\mu$ denotes the class $\eta \otimes \omega_c \in \Fil^1 H^2_{\dR, c}(X / K)$, where $\omega_c$ is any lifting of $\omega$ to compactly-supported cohomology, then
   \[ \lambda_{\mu}(r_{\syn}(z)) = [\eta; \omega, \operatorname{dlog} u].\]
  \end{proposition}

  \begin{proposition}
   Let $\eta \in H^1_{\dR, c}(Y / K)$. Let $u, v \in \mathcal{O}(Y)^\times$, and suppose that the triple symbol $[\eta; \operatorname{dlog} u, \operatorname{dlog} v]$ is defined.

   Let $z \in H^2_{\mathrm{mot}}(X, 2)$ be the cup-product of the classes $u, v \in H^1_\mathrm{mot}(Y, 1) \cong \mathcal{O}(Y)^\times$. Then
   \[ \lambda_{\eta}(r_{\syn}(z)) = [\eta; \operatorname{dlog} u, \operatorname{dlog} v].\]
  \end{proposition}

 \subsection{An alternative definition of the triple symbol}

  We conclude by giving an alternative, equivalent description of the symbol $[\eta; \omega_1, \omega_2]$, which we hope may be useful in relating our results to $p$-adic modular forms (as in the calculations of \cite{DR1} in the good-reduction case).

  The assertion that $(P_0 \star P_1 \star P_2)(q^{-1}) \ne 0$ implies that $(P_1 \star P_2)(q^{-1} \Phi^{-1})$ acts bijectively on the kernel of $P_0(\Phi)$, so that $\frac{1}{(P_1 \star P_2)(q^{-1}\Phi^{-1})} \eta$ is well-defined.

  \begin{proposition}
   Suppose that $\tilde\omega_1 \cup \tilde\omega_2 \bmod \Fil^2$ is represented by the class
   \[ [w, x, y] \in H^1_{\st, K, P_1 \star P_2}(D^1(Y_h)(2)).\]
   Then we have
   \[ [\eta; \omega_1, \omega_2] = \eta \cup y - \left(\frac{1}{(P_1 \star P_2)(q^{-1} \Phi^{-1})} \eta\right) \cup w.\]
  \end{proposition}

  \begin{proof}
   From the definition of the cup product, we see that it respects the 3-step filtration on $P$-syntomic cohomology, and the cup-products induced on the graded pieces coincide with the usual de Rham cup products. Since $\tilde\omega_1 \cup \tilde\omega_2$ lies in $\Fil^1$, and the trace isomorphism factors through $\Fil^1 / \Fil^2$, we obtain the above compatibility using \eqref{difference}.
  \end{proof}

  This takes a particularly simple form if $\eta$ is a $\Phi$-eigenvector, say $\Phi(\eta) = \alpha \eta$; then we have
  \[ [\eta; \omega_1, \omega_2] = \eta \cup \left( y - \frac{1}{(P_1 \star P_2)(q^{-1} \alpha^{-1})} w\right).\]

 \section*{Acknowledgements}

  We would like to thank Massimo Bertolini, Henri Darmon, and Victor Rotger for encouraging us to work on this project, and for several useful remarks and conversations while the paper was being written; and Jan \Ne and Wies\l awa \Ni for their patience in explaining their work \cite{NN} to us. The second and third authors would also like to thank Ehud de Shalit for his invitation to visit Jerusalem in April 2014, which gave us the opportunity to finish the paper.

\providecommand{\bysame}{\leavevmode\hbox to3em{\hrulefill}\thinspace}

\end{document}